\documentclass[10pt]{amsart}

\usepackage[
paper=a4paper,
text={138mm,203mm},centering
]{geometry}

\usepackage{amssymb}
\usepackage[all]{xy}
\usepackage{pb-diagram,pb-xy}
\usepackage{graphicx,color,float}
\usepackage[bookmarks,breaklinks]{hyperref}
\usepackage{pinlabel}
\usepackage[textwidth=1in]{todonotes}

\iftrue
\makeatletter
\def\@settitle{%
  \vspace*{-20pt}
  \begin{flushleft}%
    \baselineskip14\p@\relax
    \normalfont\bfseries\LARGE
    \@title
  \end{flushleft}%
}
\def\@setauthors{%
  \begingroup
  \def\thanks{\protect\thanks@warning}%
  \trivlist
  \large \@topsep30\p@\relax
  \advance\@topsep by -\baselineskip
  \item\relax
  \author@andify\authors
  \def\\{\protect\linebreak}%
  \authors
  \ifx\@empty\contribs
  \else
    ,\penalty-3 \space \@setcontribs
    \@closetoccontribs
  \fi
  \normalfont
  \@setaddresses
  \endtrivlist
  \endgroup
}
\def\@setaddresses{\par
  \nobreak \begingroup
  \small
  \def\author##1{\nobreak\addvspace\smallskipamount}%
  \def\\{\unskip, \ignorespaces}%
  \interlinepenalty\@M
  \def\address##1##2{\begingroup
    \par\addvspace\bigskipamount\noindent
    \@ifnotempty{##1}{(\ignorespaces##1\unskip) }%
    {\ignorespaces##2}\par\endgroup}%
  \def\curraddr##1##2{\begingroup
    \@ifnotempty{##2}{\nobreak\noindent\curraddrname
      \@ifnotempty{##1}{, \ignorespaces##1\unskip}\/:\space
      ##2\par}\endgroup}%
  \def\email##1##2{\begingroup
    \@ifnotempty{##2}{\nobreak\noindent E-mail address%
      \@ifnotempty{##1}{, \ignorespaces##1\unskip}\/:\space
      \ttfamily##2\par}\endgroup}%
  \def\urladdr##1##2{\begingroup
    \def~{\char`\~}%
    \@ifnotempty{##2}{\nobreak\noindent\urladdrname
      \@ifnotempty{##1}{, \ignorespaces##1\unskip}\/:\space
      \ttfamily##2\par}\endgroup}%
  \addresses
  \endgroup
  \global\let\addresses=\@empty
}
\def\@setabstracta{%
    \ifvoid\abstractbox
  \else
    \skip@25\p@ \advance\skip@-\lastskip
    \advance\skip@-\baselineskip \vskip\skip@
    \box\abstractbox
    \prevdepth\z@ 
    \vskip-10pt
  \fi
}
\renewenvironment{abstract}{%
  \ifx\maketitle\relax
    \ClassWarning{\@classname}{Abstract should precede
      \protect\maketitle\space in AMS document classes; reported}%
  \fi
  \global\setbox\abstractbox=\vtop \bgroup
    \normalfont\small
    \list{}{\labelwidth\z@
      \leftmargin0pc \rightmargin\leftmargin
      \listparindent\normalparindent \itemindent\z@
      \parsep\z@ \@plus\p@
      
    }%
    \item[\hskip\labelsep\bfseries\abstractname.]%
}{%
  \endlist\egroup
  \ifx\@setabstract\relax \@setabstracta \fi
}
\def\section{\@startsection{section}{1}%
  \z@{-1.2\linespacing\@plus-.5\linespacing}{.8\linespacing}%
  {\normalfont\bfseries\Large}}
\def\subsection{\@startsection{subsection}{2}%
  \z@{-.8\linespacing\@plus-.3\linespacing}{.3\linespacing\@plus.2\linespacing}%
  {\normalfont\bfseries}}
\def\subsubsection{\@startsection{subsection}{3}%
  \z@{.7\linespacing\@plus.2\linespacing}{-1.5ex}%
  {\normalfont\itshape}}
\def\@secnumfont{\bfseries}
\makeatother
\fi 

\def\to{\mathchoice{\longrightarrow}{\rightarrow}{\rightarrow}{\rightarrow}}
\makeatletter
\newcommand{\shortxra}[2][]{\ext@arrow 0359\rightarrowfill@{#1}{#2}}
\def\longrightarrowfill@{\arrowfill@\relbar\relbar\longrightarrow}
\newcommand{\longxra}[2][]{\ext@arrow 0359\longrightarrowfill@{#1}{#2}}
\renewcommand{\xrightarrow}[2][]{\mathchoice{\longxra[#1]{#2}}%
  {\shortxra[#1]{#2}}{\shortxra[#1]{#2}}{\shortxra[#1]{#2}}}
\makeatother

\makeatletter
\def\Nopagebreak{\@nobreaktrue\nopagebreak}
\makeatother

\theoremstyle{plain}
\newtheorem{theorem}{Theorem}[section]
\newtheorem{proposition}[theorem]{Proposition}

\theoremstyle{definition}

\newtheorem{remark}[theorem]{Remark}

\def\Z{\mathbb{Z}}
\def\Q{\mathbb{Q}}
\def\R{\mathbb{R}}

\def\NN{\mathbb{N}}

\def\cR{\mathcal{R}}
\def\cP{\mathcal{P}}

\def\sign{\operatorname{sign}}

\def\Spinc{\mathop{\mathrm{Spin}^c}}

\def\emptystr{}
\newcommand{\mkc}[2][]{\begin{color}{red}#2%
  \def\tempstr{#1}%
  \ifx\tempstr\emptystr \else\textsf{\SMALL\ \raise.7ex\hbox{[\tempstr]}}\fi
\end{color}}

\begin{document}

\title%
{Concordance to links with unknotted components}

\author{Jae Choon Cha}
\address{Department of Mathematics and PMI\\
 POSTECH \\
 Pohang 790--784\\
 Republic of Korea}
\email{jccha@postech.ac.kr}

\author{Daniel Ruberman}
\address{Department of Mathematics\\
  Brandeis University\\
  Waltham, MA 02454--9110\\
  USA}
\email{ruberman@brandeis.edu}

\def\subjclassname{\textup{2010} Mathematics Subject Classification}
\expandafter\let\csname subjclassname@1991\endcsname=\subjclassname
\expandafter\let\csname subjclassname@2000\endcsname=\subjclassname


\begin{abstract}
  We show that there are topologically slice links whose individual components are smoothly concordant to the
  unknot, but which are not smoothly concordant to any link with unknotted components.
   We also give generalizations in the topological category regarding components of prescribed Alexander
  polynomials.  The main tools are covering link calculus, algebraic
  invariants of rational knot concordance theory, and the correction
  term of Heegaard Floer homology.
\end{abstract}

\maketitle

\section{Introduction}

This paper addresses the following question which arises naturally in
the study of link concordance: to what extent can a component of a given link
vary under link concordance?  We have an obvious necessary
condition---for a given link $L$, if a knot $K$ appears as a component
of a link which is concordant to $L$, then $K$ is concordant to the
corresponding component of~$L$.  Therefore a natural question is
whether the converse holds: if $K$ is concordant to, say the first
component of $L$, is there a link concordant to $L$ which has $K$
as the first component?

In general, there is no such link. In the case when $K$ is unknotted, Cochran \cite{Cochran:1991-1} (using techniques introduced in~\cite{Cochran:1985-2}) and Cochran-Orr \cite{Cochran-Orr:1990-1,Cochran-Orr:1993-1} showed:
\begin{theorem}
  \label{theorem:main-top}
  There are links that have components smoothly concordant to the
  unknot but are not topologically concordant to any link with
  unknotted components.
\end{theorem}

This shows that the answer is negative in both topological and smooth
category.  In this paper, we investigate the difference between the smooth and
topological cases, and generalize Theorem~\ref{theorem:main-top} to obstruct concordances involving different knot types.  In the smooth case, we refine the
result by giving examples which are topologically trivial:

\begin{theorem}
  \label{theorem:main-smooth}
  There are topologically slice links that have smoothly slice
  components but are not smoothly concordant to any link with
  unknotted components.
\end{theorem}
In fact, there is no concordance to a link having components with trivial Alexander polynomial.

We also address the general case about concordance to links with an
arbitrarily given component, in place of an unknotted
component.
In the topological setting, we considerably extend Theorem
\ref{theorem:main-top}, giving obstructions that detect links not
concordant to any link whose first component belongs to a given
collection of Alexander polynomials: \begin{theorem}
  \label{theorem:main-topological}
  For any finite collection $D$ of classical Alexander polynomials of
  knots and for any knot $J_0$ with $\Delta_{J_0}(t)\in D$, there are
  links $L=K_1\cup K_2$ satisfying the following:
  \begin{enumerate}
  \item The first component $K_1$ of $L$ is smoothly concordant to the
    given~$J_0$.
  \item $L$ is not topologically concordant to any link $L'=K'_1\cup
    K'_2$ with $\Delta_{K'_1}(t)\in D$.
  \end{enumerate}
\end{theorem}

We give a smooth refinement of Theorem~\ref{theorem:main-topological}
by providing examples of links which are topologically concordant to a
link with first component~$J_0$. See
Theorem~\ref{theorem:topologically-slice-smooth-examples} for details.

To prove the above results, we give general obstructions by combining
several known techniques.  In particular the main ingredients are the
following: covering link calculus as used in \cite{Cochran-Orr:1993-1}
and formulated in \cite{Cha-Kim:2008-1} (see also
\cite{Cha-Livingston-Ruberman:2006-1, VanCott:2009-01,
  LevineA:2009-01}), invariants of rational knot concordance
(\cite{Cha:2003-1}; see also \cite{Cochran-Orr:1993-1,Cha-Ko:2006-1}),
and the $d$-invariant (or correction term) of Heegaard Floer homology
(\cite{Ozsvath-Szabo:2003-2}; see also \cite{Jabuka-Naik:2007-1,
  Grigsby-Ruberman-Saso:2008-1, Hedden-Livingston-Ruberman:2010-01}).

The obstructions used to prove Theorems~\ref{theorem:main-smooth} and~\ref{theorem:main-topological} are described in Section~\ref{section:obstructions}.  See (respectively)
Theorems~\ref{theorem:d-obstruction-to-given-Alexander-poly}
and~\ref{theorem:complexity-obstruction-to-given-Alexander-poly}.

In Section~\ref{section:examples}, we present explicit examples of
links.  The topological and smooth cases are dealt with in
Theorem~\ref{theorem:topological-examples} and
Theorem~\ref{theorem:topologically-slice-smooth-examples},
respectively.

In Section~\ref{section:rational-homology-cobordism}, we investigate the
nontriviality of rational homology cobordism groups modulo
the classes of $\Z_q$-homology spheres using the 3-manifolds
associated to our link examples.

\section{Links with components having given Alexander polynomials}
\label{section:obstructions}

We consider ordered links, namely the components of a link are given a
preferred labeling by integers $1,\ldots, m$.  We also always assume
that links are oriented.  We say two $m$-component links
$L=K_1\cup\cdots\cup K_m$ and $L'=K'_1\cup\cdots\cup K'_m$ in $S^3$
are \emph{topologically (resp.\ smoothly) concordant} if there are
disjoint locally flat (resp.\ smooth) cylinders $C_i$ ($i=1,\ldots,m$)
embedded in $S^3\times[0,1]$ with $\partial C_i = K_i\times 0 \cup
-K'_i \times 1$.

To simplify the notation we consider only the case of two-component
links.  We remark that all the arguments and results apply to the case
of any number of components ($\ge 2$) as well as the two-component
case.

For a link $L=K_1\cup K_2$, we consider covering links as formulated
and used in \cite{Cha-Kim:2008-1} (see also \cite{Cochran-Orr:1993-1},
\cite{Cha-Livingston-Ruberman:2006-1}, \cite{VanCott:2009-01},
\cite{LevineA:2009-01}).  In particular, we make use of the following
special case.  Suppose $d$ is a power of a prime $p$ and let
$Y_L$ (or just $Y$ if the link $L$ is understood) be
the $d$-fold cyclic branched cover of $S^3$ along $K_1$, the first
component of the given link~$L$.  It is known that $Y_L$ is a
$\Z_p$-homology sphere.  (In this paper $\Z_p$ denotes $\Z/p\Z$.)  If
$L$ has linking number zero, the pre-image of $K_2$ in $Y$ consists of
$p$ components.  Choose a component $K_L$, which we call a
\emph{covering knot} of~$L$.  (For our purpose the choice of a
component gives no ambiguity since the $\Z_d$-action permutes these
component.)

\begin{proposition}
  \label{proposition:ZHS-obstruction-to-Alexander-one-component}
  Suppose $L$ is topologically (resp.\ smoothly) concordant to a link
  $L'=K'_1\cup K'_2$ with $\Delta_{K'_1}(t)=1$.  Then for any power
  $d$ of an arbitrary prime $p$, the $d$-fold covering knot
  $(Y_L,K_L)$ of $L$ described above is topologically (resp.\
  smoothly) $\Z_p$-concordant to a knot in an integral homology
  sphere.
\end{proposition}

Here, as in \cite{Cha:2003-1}, two knots $(Y,J)$ and $(Y',J')$ in
$R$-homology 3-spheres $Y$ and $Y'$ are called topologically
(resp.~smoothly) $R$-concordant if there is a pair $(W,C)$ of
a 4-manifold $W$ and an embedded cylinder $C$ satisfying
$H_*(W,Y;R)=0=H_*(W,Y';R)$ and $\partial (W,C)=(Y,J)\cup -(Y',J')$.

\begin{proof}
  It is known that if $L$ and $L'$ are concordant, then the covering
  knot $K_L$ of $L$ is concordant to the corresponding covering knot
  $K_{L'}$ of $L'$ (e.g., see the argument of
  \cite[Theorem~2.2]{Cha-Kim:2008-1}).  The covering knot $K_{L'}$
  lies in the $d$-fold cyclic branched cover of $S^3$ along~$K_1'$,
  say $Y_{L'}$, which is an integral homology sphere, since $K'_1$ has
  Alexander polynomial one.  This argument works in both smooth and
  topological cases.
\end{proof}

A key fact used in the above proof of
Proposition~\ref{proposition:ZHS-obstruction-to-Alexander-one-component}
is that the order of the homology of $Y_{L'}$ is determined by the Alexander
polynomial of the component~$K'_1$.  Using this fact more extensively,
Proposition~\ref{proposition:ZHS-obstruction-to-Alexander-one-component}
generalizes as follows.

For convenience of notation, for a polynomial $f(t)$ we denote
\[
\cR_d(f) = \Big|\prod_{k=0}^{d-1} f(e^{2\pi k\sqrt{-1}})\Big|.
\]
It is well known that for a knot $K$ with Alexander polynomial
$\Delta_K(t)$, the first homology of the $d$-fold cyclic branched
cover of $(S^3,K)$ has order~$\cR_d(\Delta_K)$~\cite{Fox:1956-1}.  We
note that $\cR_d(\Delta_K)$ may be zero in general (in this case the
first homology has nontrivial free part), but if $d$ is a prime power,
$\cR_d(\Delta_K)$ is always a positive integer.

Throughout this paper, we will write
\[
D = \{f_1(t),\ldots,f_r(t)\}
\]
for a finite collection of classical Alexander polynomials.  By definition, these are integer coefficient polynomials satisfying $f_i(1)=\pm 1$ and $f(t^{-1}) = t^{2g}f(t)$ for some $g \in \NN$; we will make use of just the first condition.  For a
prime power $d=p^a$, define $\cP_d(D)$ to be the set of primes that
do not divide $\cR_d(f_i)$ for all~$i=1,\ldots,r$.  One easily sees
the following two properties: 
\begin{enumerate}
\item All but finitely many primes are
  in~$\cP_d(D)$.
\item For $D=\{1\}$, $\cP_d(D)$ is the set of all primes.
\end{enumerate}

\begin{theorem}
  \label{theorem:Z/q-HS-obstruction-to-given-Alexander-poly}
  Suppose a link $L$ is topologically (resp.\ smoothly) concordant to
  a link $L'=K'_1\cup K'_2$ for which $\Delta_{K'_1}(t)$ lies in~$D$.
  Then for any prime power $d=p^a$, the $d$-fold covering knot
  $(Y_L,K_L)$ of $L$ described above is topologically (resp.\
  smoothly) $\Z_p$-concordant to a knot in a 3-manifold which is a
  $\Z_q$-homology sphere for any $q \in \cP_d(D)$.
\end{theorem}

\begin{proof}
  Observe that if $\Delta_{K'_1}(t)=f_i(t)\in D$, then the $d$-fold
  covering link $K_{L'}$ of $L'$ lies in a 3-manifold $Y_{L'}$
  satisfying $|H_1(Y_{L'})|=\cR_d(f_i)$, by the above discussion.  By
  the hypothesis that $q$ does not divide $\cR_d(f_i)$, $Y_{L'}$ is a
  $\Z_q$-homology sphere. The same argument as the proof of
  Proposition~\ref{proposition:ZHS-obstruction-to-Alexander-one-component}
  concludes that $(Y_L,K_L)$ is $\Z_p$-concordant to
  $(Y_{L'},K_{L'})$.
\end{proof}

\subsection*{Obstructions from rational concordance theory: complexity
  of knots}

In \cite{Cochran-Orr:1993-1}, the notion of the complexity of a
(codimension two) knot in a rational homology sphere was first
introduced, and subsequently studied in
\cite{Cha:2003-1,Cha-Ko:2000-1} extensively. For the convenience of
the reader, we recall the definition from \cite[Definition
2.8]{Cha:2003-1}. For a knot $K$ in a rational homology 3-sphere
$\Sigma$, an Alexander duality argument shows that
$H_1(\Sigma-K;\Z)/\text{torsion}\cong \Z$. An essential difference
from the integral homology sphere case is that the class of a meridian
of $K$ does not necessarily generates
$H_1(\Sigma-K;\Z)/\text{torsion}$ though it is nonzero. The
\emph{complexity of $K$} is defined to be the absolute value of the
element in $\Z\cong H_1(\Sigma-K;\Z)/\text{torsion}$ represented by
the meridian.

Obviously, a knot in an integral homology sphere has complexity one,
by Alexander duality with integral coefficients.  Therefore, from
Proposition~\ref{proposition:ZHS-obstruction-to-Alexander-one-component},
we obtain the following:

\begin{proposition}
  [Special case of
  Theorem~\ref{theorem:complexity-obstruction-to-given-Alexander-poly}]
  \label{proposition:complexity-obstruction-to-Alexander-one-component}
  Suppose $L$ is topologically (resp.\ smoothly) concordant to a link
  $L'=K'_1\cup K'_2$ with $\Delta_{K'_1}(t)=1$.  Then for any power
  $d$ of $p$, the $d$-fold covering knot $(Y_L, K_L)$ of $L$
  described above is topologically (resp.\ smoothly)
  $\Z_p$-concordant to a knot with complexity one.
\end{proposition}

Using Theorem~\ref{theorem:Z/q-HS-obstruction-to-given-Alexander-poly}
in place of
Proposition~\ref{proposition:ZHS-obstruction-to-Alexander-one-component},
we obtain the following generalization of
Proposition~\ref{proposition:complexity-obstruction-to-Alexander-one-component}:

\begin{theorem}
  \label{theorem:complexity-obstruction-to-given-Alexander-poly}
  Suppose $L$ is topologically (resp.\ smoothly) concordant to a link
  $L'=K_1'\cup K_2'$ satisfying $\Delta_{K_1'} \in D$.  Then for any
  prime power $d=p^r$, the $d$-fold covering knot $(Y_L,K_L)$ of $L$
  is topologically (resp.\ smoothly) $\Z_p$-concordant to a knot
  whose complexity is relatively prime to all $q\in \cP_d(D)$.
\end{theorem}

\begin{proof}
  Suppose $q\in \cP_d(D)$.  By
  Theorem~\ref{theorem:Z/q-HS-obstruction-to-given-Alexander-poly},
  $K_L$ is $\Z_p$-concordant to a knot in a $\Z_q$-homology
  sphere.  The conclusion follows immediately from the following known
  fact, which is for example mentioned in \cite[p.\ 66]{Cha:2003-1}:
  the complexity of a knot $K$ in a $\Z_q$-homology sphere $\Sigma$
  is relatively prime to~$q$.  A standard argument for this is to use
  the duality with $\Z_q$ coefficients---by Alexander duality
  $H_1(\Sigma-K;\Z_q)\cong \Z_q$ is generated by a meridian of $K$,
  and thus the meridian represents an integer $\not\equiv 0 \text{
    (mod $q$)}$ in $H_1(\Sigma-K)/\text{torsion} \cong \Z$.
\end{proof}

There are known obstructions to a knot in a rational homology sphere
being topologically $\Q$-concordant to a knot with given complexity.
Cochran and Orr first discovered an obstruction from the period of a
signature function~\cite{Cochran-Orr:1993-1}, which was reformulated
in terms of Seifert matrices in~\cite{Cha-Ko:2000-1}.  Also, the first
author obtained further obstructions to being concordant to complexity
one knots from torsion invariants of rational knot concordance, in his
monograph~\cite[Theorem 4.17]{Cha:2003-1}.

In the next section we will give examples to which one can apply the
topological case of
Theorem~\ref{theorem:complexity-obstruction-to-given-Alexander-poly}
(and
Proposition~\ref{proposition:complexity-obstruction-to-Alexander-one-component})
combined with these known obstructions to being concordant to a knot
with given complexity.

\subsection*{Obstructions from Heegaard Floer $d$-invariants}

In case of smooth category, we obtain further obstructions from the
$d$-invariants which are ``correction terms'' of Heegaard Floer
homology of 3-manifolds~\cite{Ozsvath-Szabo:2003-2}.  For this purpose
first we make the following observation on the $(1/n)$-surgery along a
knot.  First note that for a knot $K$ in a rational homology sphere,
if $K$ has vanishing $(\Q/\Z)$-valued self-linking, then there is a
well-defined \emph{zero-framing} on $K$ which is uniquely determined
by the condition that $K$ and its parallel copy along the zero-framing
has vanishing $\Q$-valued linking number.  (The converse is also true.
For more about this, see \cite{Cha:2003-1}, \cite{Cha-Ko:2000-1}.)
This enables us to identify a slope of surgery along $K$ with an
element in $\Q\cup \{\infty\}$ when $K$ has vanishing $(\Q/\Z)$-valued
self-linking, exactly as in case of knots in~$S^3$.  In particular the
$(a/b)$-surgery along $K$ is well-defined for $a/b\in \Q$.

We remark that if a component, say $K$, of a link in a
$\Z_p$-homology sphere has vanishing $(\Q/\Z)$-valued self-linking,
then any component of a covering link of $L$ that projects to $K$ has
vanishing $(\Q/\Z)$-valued
self-linking~\cite[Section~2]{Cha-Kim:2008-1} (see also
\cite{Cha:2003-1}).  In particular, for a link $L$ in $S^3$, our covering
knot $(Y_L,K_L)$ has vanishing $(\Q/\Z)$-valued self-linking.  It
follows that the $(a/b)$-surgery along $K_L$ is well-defined.

\begin{proposition}
  \label{proposition:1/n-surgery-along-knots-in-ZHS}
  Let $R$ be a subring of $\Q$ or~$\Z_p$.  Suppose $K$ is a knot in an
  $R$-homology sphere with vanishing $(\Q/\Z)$-valued self-linking.
  If $K$ is $R$-concordant to a knot in a $\Z_q$-homology sphere,
  then the $(a/b)$-surgery manifold of $K$ is $R$-homology cobordant
  to a $\Z_q$-homology sphere for any $a$ relatively prime to~$q$.
\end{proposition}

\begin{proof}
  We may assume $R$ is a subring of $\Q$, since we can replace $\Z_p$
  with $\Z_{(p)}$, the ring of integers localized at~$p$.  The
  rational valued linking number is invariant under rational
  concordance (e.g., see \cite{Cha-Ko:2000-1,Cha:2003-1}).  Therefore,
  if $(W,C)$ is an $R$-concordance between two knots, then (compare
  \cite{Gordon:1975-1}) by removing a tubular neighborhood of $C$ from
  $W$ and filling in it with $S^1\times D^2\times[0,1]$ along the
  $(a/b)$-framing, we obtain an $R$-homology cobordism between their
  $(a/b)$-surgery manifolds.  The proof is completed by observing that
  if a knot is in a $\Z_q$-homology sphere and $(a,q)=1$, then the
  $(a/b)$-surgery along the knot is again a $\Z_q$-homology sphere.
\end{proof}

It was observed in~\cite{Hedden-Livingston-Ruberman:2010-01} that the
Ozsv\'ath-Szab\'o $d$-invariant~\cite{Ozsvath-Szabo:2003-2} gives an
obstruction to being smoothly $\Q$-homology cobordant to an integral
homology sphere.  To simplify the discussion of $\Spinc$ structures,
we restrict to the case of $\Z_2$-homology spheres and homology
cobordisms.  For a $\Z_2$-homology 3-sphere $Y$, the composition
\[
\Spinc(Y) \xrightarrow{c_1} H^2(Y) \xrightarrow{\mathrm{PD}} H_1(Y)
\]
of the first Chern class and Poincar\'e duality induces a bijection
that takes the unique Spin structure on $Y$ to $0 \in H_1(Y)$.  In the
remainder of the paper, we use this bijection to label $\Spinc$
structures by elements of $H_1(Y)$.

For $s\in H_1(Y)$, denote by $d(Y,s)$ the correction term invariant
defined from the $\Q$-valued grading of the Heegaard Floer homology of
$Y$.  Let
\[
\bar d(Y,s)=d(Y,s)-d(Y,0).
\]

Recall that a torsion abelian group $G$ decomposes into an inner direct
sum $G=\bigoplus_p G_p$ of $p$-primary summands
\[
G_p = \{x\in G \mid p^N x = 0 \text{ for some } N\ge 0\}
\]
where $p$ runs over all primes.  We remark that if $G$ is finitely
generated, that $|G_p|$ is equal to the maximal power of $p$ dividing
the order of~$G$.  Also, if $A$ is a subgroup of $G_p$, then obviously
$A_p = A\cap G_p$.

For our purpose the following generalization of
\cite[Theorem~3.2]{Hedden-Livingston-Ruberman:2010-01} is useful.  For
similar applications of $d$ and $\bar d$, see, e.g.,
\cite{Jabuka-Naik:2007-1, Grigsby-Ruberman-Saso:2008-1}.

\begin{proposition}
  \label{proposition:obstruction-to-being-H-cob-to-(Z/q)HS}
  Suppose $q$ is a prime and $Y$ is a $\Z_2$-homology 3-sphere which
  is $\Z_2$-homology cobordant to a $\Z_q$-homology sphere.  Then
  there is a subgroup $H$ of the $q$-primary part $H_1(Y)_q$ satisfying $|H|^2 =
  |H_1(Y)_q|$ and $\bar d(Y,s)=0$ for any $s\in H$.
\end{proposition}

We note that
Proposition~\ref{proposition:obstruction-to-being-H-cob-to-(Z/q)HS}
easily applies to a 3-manifold $Y$ which is a boundary component of a
$\Z_2$-homology punctured 4-ball $W$ with $H_1(\partial
W-Y;\Z_q)=0$, by tunneling $W$ to join components of $\partial W-Y$.

\begin{proof}
  Suppose $Y'$ is a $\Z_q$-homology sphere and $Y$ is
  $\Z_2$-homology cobordant to~$Y'$.  Denote $G=H_1(Y)\oplus
  H_1(Y')=H_1(\partial W)$.  Let $A$ be the image of the boundary map
  $H_2(W,\partial W) \to G$ which is identified with $\Spinc(W)=H^2(W)
  \to H^2(\partial W)=\Spinc(\partial W)$ under Poincare duality.
  Namely, $A$ consists of the $\Spinc$-structures of $\partial W$ that
  extends to~$W$.  By Theorem 1.2 in 
  Ozsv\'ath-Szab\'o~\cite{Ozsvath-Szabo:2003-2}, $d(\partial W,(s,s'))
  = d(Y,s)-d(Y',s') =0$ whenever $(s,s')\in A \subset G$.  Also, by
  the argument of Casson-Gordon~\cite{Casson-Gordon:1986-1}, $|A|^2 =
  |G|$.

  Consider the $q$-primary parts $G_q$ and~$A_q$.  Since $|H_1(Y')|$
  is relatively prime to $q$, we have $G_q=H_1(Y)_q$ and $A_q \subset
  G_q \subset H_1(Y)$.  Also, $|A|^2=|G|$ implies $|A_q|^2 =
  |H_1(Y)_q|$, by using the remark before
  Proposition~\ref{proposition:obstruction-to-being-H-cob-to-(Z/q)HS}.

  Now, for $s\in A_q\subset H_1(Y)$, since $(s, 0) \subset A \subset
  G$, we have $d(Y,s)-d(Y',0)=0$.  In particular $d(Y,0)-d(Y',0)=0$.
  It follows that $\bar d(Y,s)=d(Y,s)-d(Y,0)=0$ for any $s\in
  H_1(Y)_q$.
\end{proof}

As remarked above, if $d$ is a power of $2$, then the $d$-fold branched cover of $S^3$ 
along a knot is a $\Z_2$-homology sphere.  Hence we may combine 
Theorem~\ref{theorem:Z/q-HS-obstruction-to-given-Alexander-poly} and
Propositions~\ref{proposition:1/n-surgery-along-knots-in-ZHS} and
\ref{proposition:obstruction-to-being-H-cob-to-(Z/q)HS}, to obtain the
following:

\begin{theorem}
  \label{theorem:d-obstruction-to-given-Alexander-poly}
  Suppose $L$ is smoothly concordant to a link $L'=K_1'\cup K_2'$
  satisfying $\Delta_{K_1'}(t) \in D$.  Suppose $Y$ is the
  $(a/b)$-surgery manifold of the $d$-fold covering knot $(Y_L,K_L)$
  of~$L$.  If $d$ is a power of two, $q\in \cP_d(D)$, and $a$ is
  relatively prime to $q$, then there is a subgroup $H$ of the
  $q$-primary part $\subset H_1(Y)_q$ such that $|H|^2 = |H_1(Y)_q|$
  and $\bar d(Y,s)=0$ for any $s\in H$.
\end{theorem}

We remark that if $D=\{1\}$, namely if $L$ is assumed to be concordant
to $L'$ with $\Delta_{K_1'}(t)=1$ in
Theorem~\ref{theorem:d-obstruction-to-given-Alexander-poly}, then the
condition ``$q\in \cP_d(D)$'' is replaced with ``$q$ is any prime'' in
the conclusion.

We also remark that our subgroup $H$ in
Theorem~\ref{theorem:d-obstruction-to-given-Alexander-poly} has the
property that the $(\Q/\Z)$-valued linking form of $Y$ vanishes on
$H\times H$.  In this paper we do not use this.

\section{Examples}
\label{section:examples}

Our main examples are of the following form.  Fix a knot~$J_0$.  For
an integer $m$ and a knot $J$, let $L(m,J)$ be the link illustrated in
Figure~\ref{figure:colink}.  Here $J$ is a knot that will be specified
later, and \fbox{$-m$} denotes negative $m$ full twists between the
obvious two vertical bands, in such a way that no self-twisting is
added on each band.  We remark that in the case of unknotted $J_0$,
the link $L(m,J)$ was first considered in \cite{Cochran-Orr:1993-1} to
give examples which are not concordant to boundary links.  In
\cite{Cochran-Orr:1993-1} it is shown that $L(m,J)$ is a homology
boundary link for unknotted~$J_0$.  The same method shows that
$L(m,J)$ is a homology boundary link for any~$J_0$.  Consequently
$L(m,J)$ has vanishing $\bar\mu$-invariants.

\begin{figure}[h]
  \begin{center}
    \labellist
    \normalsize\hair 0mm
    \pinlabel {$-m$} at 82 100
    \pinlabel {$J$} at 140 100
    \pinlabel {$K_1$} at 177 32
    \pinlabel {$K_2$} at -6 100
    \pinlabel {$J_0$} at 83 15
    \pinlabel {\footnotesize$\alpha$} at 35 35
    \pinlabel {\footnotesize$\beta$} at 130 33
    \endlabellist
    \includegraphics[scale=.7]{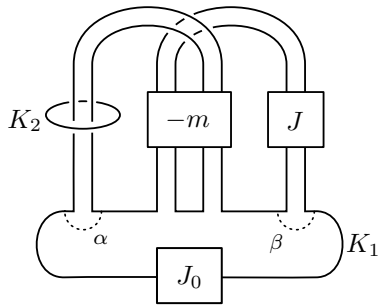}
  \end{center}
  \caption{The link $L(m,J)$.}
  \label{figure:colink}
\end{figure}

\begin{figure}[h]
 \begin{center}
   \labellist
   \small\hair 0mm
   \pinlabel {$-2m-1$} at 68 49
   \normalsize
   \pinlabel {$J\#J^r$} at 121 49
   \pinlabel {$K_{L(m,J)}$} at -8 38
   \pinlabel {$0$} at 61 4
   \pinlabel {$0$} at 121 4
   \endlabellist
   $\vcenter{\hbox{\includegraphics[scale=.9]{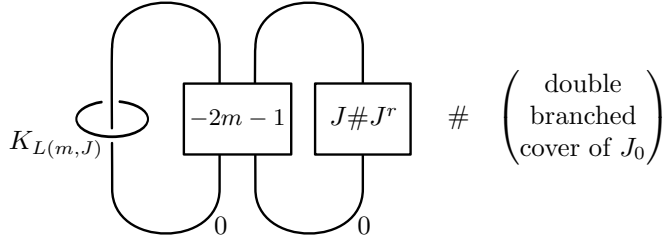}}}$
   \quad$\#$\quad$\left(
     \begin{tabular}{@{}c@{}}
       double \\
       branched\\
       cover of $J_0$
     \end{tabular}
   \right)$
 \end{center}
 \caption{The covering knot $K_{L(m,J)}$ in $Y_{L(m,J)}$}
 \label{figure:cocovering}
\end{figure}

Figure~\ref{figure:cocovering} illustrates a surgery diagram of the
2-fold covering knot $(Y_{L(m,J)}, K_{L(m,J)})$, which is obtained
from Figure~\ref{figure:colink} by applying the technique of Akbulut
and Kirby~\cite{Akbulut-Kirby:1979-1}.  The ambient space $Y_{L(m,J)}$
is the connected sum of the double branched cover of $J_0$ and the
double branched cover of~$K_1$.  The latter is the result of surgery
on $S^3$ along the two 0-framed circles in
Figure~\ref{figure:cocovering}.

Note that $K_1$ and $J_0$ are smoothly concordant, as knots.  One way
to see this is as follows: if one performs a band surgery on $K_1$
along the dotted arc $\beta$ in Figure~\ref{figure:colink} (and
forgets $K_2$), then $K_1$ becomes a split union of $J_0$ and an
unknotted circle.

\subsection{Topological examples}

In this subsection we prove the following result:

\begin{theorem}
  \label{theorem:topological-examples}
  For any finite collection $D$ of classical Alexander polynomials of
  knots and for any knot $J_0$ with $\Delta_{J_0}(t)\in D$, there are
  links $L=K_1\cup K_2$ satisfying the following:
  \begin{enumerate}
  \item The first component $K_1$ of $L$ is smoothly concordant to the
    given~$J_0$.
  \item $L$ is not topologically concordant to any link $L'=K'_1\cup
    K'_2$ with $\Delta_{K'_1}(t)\in D$.
  \end{enumerate}
\end{theorem}

Note that our example $L(m,J)$ in Figure \ref{figure:colink} satisfies
Theorem~\ref{theorem:topological-examples}~(1) as observed above.  In
what follows we will show that for an appropriate choice of $J$ that
will be given later $L(m,J)$ satisfies
Theorem~\ref{theorem:topological-examples}~(2).

We recall that for a knot $K$ in a rational homology sphere $Y$ with
vanishing $(\Q/\Z)$-valued self-linking, there is a signature
invariant of~$K$ \cite{Cochran-Orr:1993-1,Cha-Ko:2000-1,Cha:2003-1}.
For readers who are not familiar with this, we give a description in
terms of Seifert matrices, which was first given
in~\cite{Cha-Ko:2000-1}.  (The description below differs from that in~\cite{Cha-Ko:2000-1} by a reparametrization by a factor of~$2\pi$.)

It is known that if $K$ has vanishing $(\Q/\Z)$-valued linking, then
there is a \emph{generalized Seifert surface} for $K$, which is
defined to be an properly embedded surface $F$ in the exterior of $K$
whose boundary consists of $c$ parallel copies of $K$ taken along the
$0$-framing.  It is known that there exists such $F$ if and only if
$c$ is a multiple of the complexity of~$K$ (see \cite{Cha-Ko:2000-1},
\cite[Chapter~2]{Cha:2003-1}).  A Seifert matrix $A$ for the
generalized Seifert surface $F$ is defined as usual, using the
$\Q$-valued linking number in $Y$ in place of the ordinary linking
number.  Then for $\theta\in \R$, $\delta_K(\theta)$ is defined to be
the jump at $t=\theta/c$ of the signature function
\[
\sigma_A(t)= \sign\big((1-e^{2\pi t\sqrt{-1}})A+(1-e^{-2\pi
  t\sqrt{-1}})A^T\big).
\]
We regard $\delta_K$ as a function $\R\to \Z$.  It is known that
$\delta_K$ is an invariant under rational concordance
\cite{Cochran-Orr:1993-1,Cha-Ko:2000-1,Cha:2003-1}.  We note that the
reparametrization by the factor of $1/c$ given in the definition is
essential in proving the invariance.

For our purpose the following property is useful.  We say that
$\delta_K$ has period $c$ if $\delta_K(\theta)=\delta_K(\theta+c)$ for
all~$\theta$. 

\begin{theorem}[\cite{Cochran-Orr:1993-1,Cha-Ko:2000-1,Cha:2003-1}]
  \label{theorem:complexity-and-signature-period}
  If a knot $K$ in a rational homology sphere is rationally concordant
  to a knot with complexity $c$, then $\delta_K(\theta)$ has
  period~$c$.
\end{theorem}

We remark that since the set $\{\theta\in \R\mid \delta_K(\theta)\ne
0\}$ of non-vanishing points is discrete, there is a minimal period,
say $c_0>0$, of $\delta_K(\theta)$ such that any period of
$\delta_K(\theta)$ is an integer multiple of~$c_0$.  (For, otherwise
the set of periods of $\delta_K$ is dense in $\R$.)

\begin{proof}[Proof of Theorem~\ref{theorem:topological-examples}]
  As observed in the beginning of this section, the first component
  $K_1$ of our $L(m,J)$ is concordant to~$J_0$, regardless of the
  choice of~$J$.  This shows $(1)$.

  Choose a knot $J$ whose signature jump function $\delta_J(\theta)$
  has minimal period one.  For example, one can take the trefoil knot
  $3_1$ as $J$, since $\delta_{3_1}(\theta)=\pm 2$ exactly for
  $\theta\equiv \pm 1/6 \pmod{1}$.  We note that many more such knots
  exist, e.g., by appealing to \cite{Cha-Livingston:2002-1}.

  Choose an odd prime $q\in \cP_2(D)$ and let $m=(q-1)/2$.  Consider
  the link $L=L(m,J)$ and the 2-fold covering knot $(Y_L,K_L)$ of~$L$.
  The realization argument of \cite[Section 4.1.2]{Cha:2003-1} (see
  also \cite[Example, p.~1179]{Cha-Ko:2000-1}) shows that
  $\delta_{K_L}(\theta)=\delta_{J\#J^r}(\theta/q)$.  Since
  $\delta_{J\#J^r}(\theta) = \delta_{J}(\theta)+\delta_{J^r}(\theta) =
  2\delta_{J}(\theta)$, it follows that $\delta_{K_L}(\theta)$ has
  minimal period~$q$.

  Therefore, no period of $\delta_{K_L}(\theta)$ is relatively
  prime to $q$.  By
  Theorem~\ref{theorem:complexity-and-signature-period}, $K_L$ is not
  rationally concordant to any knot whose complexity is relatively
  prime to~$q$. By
  Theorem~\ref{theorem:complexity-obstruction-to-given-Alexander-poly}
  it follows that $L$ is not concordant to any link $L'=K'_1\cup K'_2$
  for which $\Delta_{K'_1}(t) \in D$.
\end{proof}

Theorem~\ref{theorem:main-top} stated in the introduction is an
immediate consequence of the special case of
Theorem~\ref{theorem:topological-examples} for $D=\{1\}$ and $J_0=$
unknot.

\begin{remark}
  Another class of links that may be shown to satisfy
  Theorem~\ref{theorem:main-top} by our method is given in
  \cite[Figure~2]{Cha-Ko:2006-1}.  In the proof of
  \cite[Theorem~3.3]{Cha-Ko:2006-1}, it was shown that these links
  have covering links which are not concordant to a complexity one
  knot.  Therefore, our method shows that these satisfy
  Theorem~\ref{theorem:main-top}.  An interesting property of the
  links given in \cite[Figure~2]{Cha-Ko:2006-1} is that these are
  mutants of ribbon links.  Therefore these links are not
  distinguished from smoothly slice links by any invariants preserved
  under mutation.
\end{remark}

\begin{remark}
  After we had announced our main results, Charles Livingston informed
  us that he found an alternative approach using Casson-Gordon
  invariants as in \cite{Livingston:1990-1} to show
  Theorem~\ref{theorem:main-top} for the same
  examples.
\end{remark}

\subsection{Topologically slice smooth examples}

\begin{theorem}
  \label{theorem:topologically-slice-smooth-examples}
  For any finite collection $D$ of classical Alexander polynomials of
  knots and for any knot $J_0$ with $\Delta_{J_0}(t)\in D$, there are
  links $L=K_1\cup K_2$ satisfying the following:
  \begin{enumerate}
  \item $L$ is topologically concordant to the split union of $J_0$
    and an unknotted circle.
  \item The first component $K_1$ of $L$ is smoothly concordant to the
    given~$J_0$.
  \item $L$ is not smoothly concordant to any link $L'=K'_1\cup
    K'_2$ with $\Delta_{K'_1}(t)\in D$.
  \end{enumerate}
\end{theorem}

\begin{proof}
  Again we consider $L=L(m,J)$, now with the extra condition that $J$
  be topologically slice.  As before, $L$ satisfies~(2).  Also note
  that $L$ satisfies (1), since one obtains the split union $J_0 \cup
  (\text{unknot}) \cup K_2$ from $L(m,\text{unknot})$ by a band
  surgery on $K_1$ along the dotted arc $\alpha$ in
  Figure~\ref{figure:colink}.

  Let $Y=Y(m,J)$ be the 3-manifold obtained by performing $1$-surgery
  on $Y_{L}$ along $K_{L}$.  Recall that we need to be careful with
  the framing: the surgery coefficient is determined with respect to
  the zero-framing on $K_{L}$ in $Y_{L}$, which is defined in terms of
  the $\Q$-valued linking number in $Y_{L}$, as in the previous
  section.  In our case the zero-framing on $K_{L}$ in $Y_{L}$ is
  identical with the usual zero-framing on $K_{L}$ in the surgery
  description of $Y_{L}$ given in Figure~\ref{figure:cocovering}.
  This can be verified, for example, by computing the $\Q$-valued
  linking number of $K_{L}$ and its preferred longitude using a
  formula given in \cite[Theorem~3.1]{Cha-Ko:2000-1}.  Alternatively,
  a surface in $Y_{L}$ which is bounded by $(2m+1)$ parallel copies of
  $K_{L}$ is obtained from a capped-off Seifert surface for $J\#J^r$
  by attaching $(2m+1)$ tubes, and then the framing induced by this
  surface, which is the zero-framing on $K_{L}$ in $Y_{L}$, is seen to
  be equal to the ordinary zero-framing.
  
  Therefore Figure~\ref{figure:cocovering} with surgery coefficient 1
  on $K_{L}$ becomes a surgery diagram for the $1$-surgery
  manifold~$Y$.  Let $M_0$ be the double branched cover of~$J_0$, and
  let $M$ be the $(2m+1)^2$-surgery manifold of the knot
  $T(2m+1,2m)\#J\#J^r$.  Here $T(a,b)$ denotes the $(a,b)$-torus knot.
  As illustrated in Figure~\ref{figure:1-surgery}, one sees that $Y$
  is diffeomorphic to $M_0 \# M$.

  \begin{figure}[H]
    \begin{center}
      \labellist
      \small\hair 0mm
      \pinlabel {$-2m-1$} at 63 49
      \pinlabel {$1$} at 3 59
      \normalsize\pinlabel {$J\#J^r$} at 116 49
      \small
      \pinlabel {$0$} at 55 4
      \pinlabel {$0$} at 115 4
      \normalsize\pinlabel {$J\#J^r$} at 263 60
      \small
      \pinlabel {$-1$} at 214 4
      \pinlabel {$0$} at 245 87
      \normalsize\pinlabel {$J\#J^r$} at 398 56
      \small
      \pinlabel {$+1$} at 327 46
      \pinlabel {$(2m+1)^2$} at 375 8
      \pinlabel {\begin{tabular}{c}$2m+1$\\strands\end{tabular}} at 302 16
      \endlabellist
      \includegraphics[scale=.9]{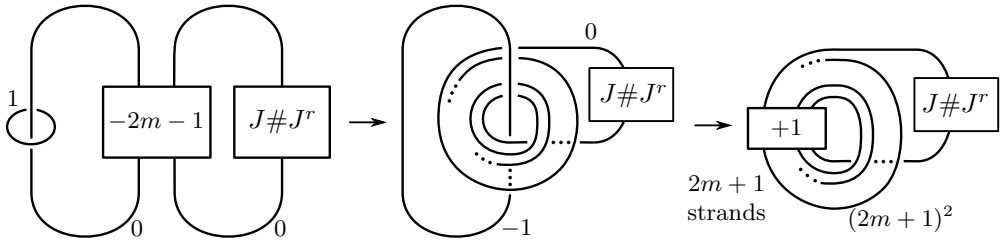}
    \end{center}
    \caption{1-surgery along $K_{L}$.}
    \label{figure:1-surgery}
  \end{figure}

  Now we specify $m$ and $J$.  For the given collection $D$, choose an
  odd prime $q\in \cP_2(D)$, and let $m=(q-1)/2$.  Let $J$ be the
  connected sum of $(3m+1)/2$ copies of the positive Whitehead double
  of the right-hand trefoil knot.  The knot $J$ is topologically slice
  by Freedman's theorem~\cite{Freedman:1984-1}.

  It is easily seen that $H_1(Y)=H_1(M_0)\oplus H_1(M)$ and
  $H_1(M)=(\Z_{2m+1})^2=(\Z_q)^2$.  Since $q\in \cP_2(D)$ and
  $\Delta_{J_0}(t)\in D$, $|H_1(M_0)|$ is relatively prime to~$q$ by
  our definition of $\cP_2(D)$ (see the discussion in
  Section~\ref{section:obstructions}).  Therefore $H_1(Y)_q = H_1(M)$,
  and the subgroup $H=q\Z_{q^2}$ generated by $q\in H_1(M)=\Z_{q^2}$ is
  the unique subgroup satisfying $|H|^2=|H_1(Y)_q|$.  By
  Theorem~\ref{theorem:d-obstruction-to-given-Alexander-poly}, for any
  $s\in H \subset H_1(Y)$, we have $\overline d(Y,s)=0$ if $L$ is
  concordant to a link $L'=K'_1\cup K'_2$ satisfying $\Delta_{K'_1}(t)
  \in D$.  Note that $\overline d(Y,s)=\overline d(M,s)$ since $s$
  lies in $H_1(M) \subset H_1(M_0)\oplus H_1(M)=H_1(Y)$.

  In \cite[Theorem~6.4]{Hedden-Livingston-Ruberman:2010-01}, Hedden,
  Livingston and the second author proved that for the above $J$,
  $\overline d(M,q)\ge 2$.  Since $q\in H$, the conclusion follows
  from this.
\end{proof}

Theorem~\ref{theorem:main-smooth} in the introduction is an immediate
consequence of the special case of
Theorem~\ref{theorem:topologically-slice-smooth-examples} for
$D=\{1\}$ and $J_0=$ unknot.

\section{Rational homology cobordism group and $\Z_p$-homology spheres}
\label{section:rational-homology-cobordism}

In \cite{Hedden-Livingston-Ruberman:2010-01}, a calculation of
$d$-invariants is used to elucidate the structure of smooth rational
homology cobordism group modulo integral homology spheres.  In this
section, we investigate the more general case of $\Z_p$-homology
spheres instead of integral homology spheres.  For this purpose we use
Proposition~\ref{proposition:obstruction-to-being-H-cob-to-(Z/q)HS},
which gives obstructions to being rational homology cobordant to
$\Z_p$-homology spheres.

For simplicity we first consider the case of $\Z_2$-homology spheres.
Let $\Omega$ be the group of smooth rational homology cobordism
classes of $\Z_2$-homology spheres.  Let $\Omega_T$ be the kernel of
the natural homomorphism of $\Omega$ into its topological analogue,
namely, $\Omega_T$ is the subgroup in $\Omega$ that consists of all
$\Z_2$-homology spheres which are topologically rational homology
cobordant to $S^3$.  For a commutative ring $R$, let
$\Omega_T^{R}$ be the subgroup of $\Omega_T$ generated by (the classes
of) $R$-homology spheres in~$\Omega_T$.

\begin{theorem}
  \label{theorem:rational-homology-cobordism-mod-Z_q-spheres}
  For any odd prime $q$, $\Omega_T / \Omega^{\Z_q}_T$ is nontrivial.
\end{theorem}

\begin{proof}
  In the proof of
  Theorem~\ref{theorem:topologically-slice-smooth-examples}, we
  observed that the $\Z_2$-homology sphere $M$ has the following
  properties: (i) The subgroup $H$ generated by $q\in H_1(M)=H_1(M)_q$
  is the unique one satisfying $|H|^2 = |H_1(M)_q|$, and (ii)
  $\overline{d}(M,q)\ge 2$.  By
  Proposition~\ref{proposition:obstruction-to-being-H-cob-to-(Z/q)HS},
  it follows that $M$ is not rational homology cobordant to any
  $\Z_q$-homology sphere.
\end{proof}

\begin{remark}
  We can also think of, in place of $\Omega$, the smooth spin rational
  homology cobordism group of spin rational homology 3-spheres, and
  spin analogues of the subgroups $\Omega_T$ and $\Omega_T^{\Z_q}$.
  Our argument above shows that
  Theorem~\ref{theorem:rational-homology-cobordism-mod-Z_q-spheres}
  also holds for this case as well, since one can canonically identify
  $\Spinc(Y)$ with $H^2(Y)$ for manifolds~$Y$ with a chosen spin structure.
\end{remark}

As a corollary we obtain that $\Omega_T/\Omega_T^{\Z}$ is nontrivial,
which is a consequence of
\cite[Theorem~7.1]{Hedden-Livingston-Ruberman:2010-01}.  In fact
\cite[Theorem~7.1]{Hedden-Livingston-Ruberman:2010-01} says that
$\Omega_T/\Omega_T^{\Z}$ has infinite rank.  Regarding this and our
theorem above, one may ask the following question: for an odd
prime~$q$, does $\Omega_T/\Omega_T^{\Z_q}$ have infinite rank?

\subsubsection*{Acknowledgments} We thank Tim Cochran and Kent Orr for
informative comments on the first version of this paper, and for
providing references to earlier work in the area.  We also thank
Charles Livingston for his comments.  The first author was supported
by the National Research Foundation of Korea (NRF) grant funded by the
Korea government (MEST), No.\ 2010--0011629 and 2010--0029638.  The
second author was partially supported by NSF Grants DMS-0804760 and DMS-1105234, and by JSPS Grant-in-Aid No.\ 19340015.

\def\cprime{$'$}

\end{document}